\documentclass[11pt]{amsart}%
\usepackage{hyperref}
\usepackage{amsmath}%
\usepackage{amsfonts}%
\usepackage[left=1.5in, right=1.5in, top=1.5in, bottom=1.5in]{geometry}
\usepackage{amssymb}%
\usepackage{amsthm}
\usepackage{array}
\usepackage{mathrsfs}
\usepackage{rotating}
\usepackage{comment}
\usepackage{todonotes}
\usepackage{bbm}
\usepackage{amscd}
\usepackage{youngtab}
\usepackage{tikz}
\usepackage{bm}
\usepackage{enumitem}
\usetikzlibrary{arrows}
\usetikzlibrary{matrix}

\usetikzlibrary{cd}

\newcommand{\Q}{\mathbb{Q}}

\newcommand{\Y}{\mathbb{Y}}
\newcommand{\N}{\mathbb{N}}

\DeclareMathOperator{\End}{End}
\DeclareMathOperator{\Par}{Par}
\DeclareMathOperator{\Id}{Id}
\DeclareMathOperator{\tr}{tr}

\newtheorem{theorem}{Theorem}[section]
\newtheorem{def-prop}[theorem]{Definition-Proposition}
\newtheorem{prop}[theorem]{Proposition}

\newtheorem{lemma}[theorem]{Lemma}
\newtheorem{cor}[theorem]{Corollary}

\theoremstyle{definition}
\newtheorem{ex}[theorem]{Example}
\newtheorem{defin}[theorem]{Definition}

\theoremstyle{remark}
\newtheorem*{remark}{Remark}

\begin{document}

\title{Stable characters from permutation patterns}
\author{Christian Gaetz}
\thanks{C.G. is supported by a National Science Foundation Graduate Research Fellowship under Grant No. 1122374.}
\address{Department of Mathematics, Massachusetts Institute of Technology, Cambridge, MA.}
\email{\href{mailto:gaetz@mit.edu}{gaetz@mit.edu}} 

\author{Christopher Ryba}
\email{\href{mailto:ryba@mit.edu}{ryba@mit.edu}} 
\date{\today}

\begin{abstract}
For a fixed permutation $\sigma \in S_k$, let $N_{\sigma}$ denote the function which counts occurrences of $\sigma$ as a pattern in permutations from $S_n$.  We study the expected value (and $d$-th moments) of $N_{\sigma}$ on conjugacy classes of $S_n$ and prove that the irreducible character support of these class functions stabilizes as $n$ grows.  This says that there is a single polynomial in the variables $n, m_1, \ldots, m_{dk}$ which computes these moments on any conjugacy class (of cycle type $1^{m_1}2^{m_2}\cdots$) of any symmetric group.  This result generalizes results of Hultman \cite{Hultman2014} and of Gill \cite{Gill2013}, who proved the cases $(d,k)=(1,2)$ and $(1,3)$ using ad hoc methods.  Our proof is, to our knowledge, the first application of partition algebras to the study of permutation patterns.
\end{abstract}

\maketitle

\section{Introduction} \label{sec:intro}

A permutation $\pi=\pi_1 \cdots \pi_n$ in the symmetric group $S_n$ is said to \emph{contain the pattern} $\sigma=\sigma_1 \cdots \sigma_k \in S_k$ if there exist indices $1 \leq i_1 < \cdots < i_k \leq n$ such that $\pi_{i_a}<\pi_{i_b}$ if and only if $\sigma_a < \sigma_b$;  
in this case we say $(i_1,\ldots,i_k)$ is an \emph{occurrence} of $\sigma$.  If $\pi$ does not contain any occurrences of $\sigma$, then it is said to \emph{avoid} $\sigma$.  

One of the earliest prominent results in the theory of permutation patterns was Knuth's \cite{Knuth} characterization of ``stack-sortable" permutations as those avoiding the pattern $\sigma=231$.  Since then, permutation patterns have been found to play an important role in many settings where algebraic or geometric objects are indexed by permutations, being ubiquitous in the study of Schubert varieties, Bruhat order, and Kazhdan--Lusztig polynomials \cite{Abe-Billey}.  The study of permutation patterns has grown into a very active field in its own right.

We write $N_{\sigma}(\pi)$ for the number of occurrences of $\sigma$ in $\pi$, and view
\[
N_{\sigma}: \coprod_{n \geq 1} S_n \to \N
\]
as a function on all symmetric groups at once.  These pattern counting functions are very well studied (see, e.g. \cite{Bona-textbook}).  In particular, it was shown by Janson, Nakamura, and Zeilberger \cite{Janson-Nakamura-Zeilberger} that the distribution of $N_{\sigma}$ on a uniformly random permutation from $S_n$ is asymptotically normal, and by Zeilberger that the moments of this distribution are given by polynomials in $n$ \cite{Zeilberger}.  Theorem~\ref{thm:main} refines this polynomiality result by showing that there is a polynomial which computes these moments for all conjugacy classes of all symmetric groups; Corollary~\ref{cor:recover-zeilberger} recovers Zeilberger's result.  

One area of recent interest is the, as yet poorly understood, interaction between permutation patterns and the group structure of symmetric groups, which has proven quite difficult to understand.  For example, Richard Stanley asked in 2007 for an enumeration of cyclic permutations in $S_n$ avoiding a given pattern from $S_3$; despite the fact that the total number of avoiders for each of these patterns is well-known to be given by the Catalan numbers, only partial progress has been made on Stanley's question \cite{Archer-Elizalde, Bona-Cory} with none of the cases completely resolved.  Other examples include the study of permutation patterns in involutions \cite{invol2, invol1} and in powers of a given permutation \cite{Bona-Smith}.  

Hultman \cite{Hultman2014} observed how turning permutation statistics (such as $N_{\sigma}$) into virtual $S_n$-characters allows for computation of their expected values on random permutations drawn from various distributions.  Computing these characters directly requires understanding occurrences of $\sigma$ in every conjugacy class, a problem which is considered difficult by experts.  Hultman and later Gill \cite{Gill2013} were able to compute these characters for $\sigma \in S_2$ and $\sigma \in S_3$, respectively, via some case work, but this direct enumeration is clearly infeasible for longer patterns $\sigma$ and for higher moments $N^d_{\sigma}$ of the pattern counting functions.  Theorem~\ref{thm:main} and its proof using partition algebras give a uniform framework for understanding these characters for all $\sigma$ and for all $d$.  In particular we give a conceptual explanation for the curious fact, noted by Hultman and Gill, that the irreducible character supports in the cases they computed are surprisingly small, and stabilize as $n$ grows.    

\subsection{Permutation pattern polynomials}
\label{sec:intro-ppp}

For a fixed permutation $\sigma \in S_k$, and positive integers $d,n$, we define a class function $M_{\sigma, d, n}: S_n \to \N$ by 
\begin{equation}
\label{eq:M-def}
M_{\sigma,d,n}(\pi)=\frac{1}{|C_{\pi}|} \sum_{\pi' \in C_{\pi}} N_{\sigma}^d(\pi'),
\end{equation}
where $C_{\pi}$ denotes the conjugacy class of $\pi \in S_n$.  That is, $M_{\sigma,d,n}(\pi)$ is the (raw) $d$-th moment of the random variable $N_{\sigma}$ when permutations are chosen uniformly at random from $C_{\pi}$.  In the case $d=1$ and $k=2,3$ this function has previously been studied by Hultman \cite{Hultman2014} and by Gill \cite{Gill2013}.  Conjugacy classes in the symmetric group are determined by the cycle types of permutations; for $i=1,2,\ldots$ we write $m_i$ for the function on $\coprod_{n \geq 1} S_n$ taking $\pi$ to the number of $i$-cycles in its cycle decomposition.

By definition $M_{\sigma, d, n}$ is a class function on $S_n$, and thus it may be uniquely written as a linear combination of the irreducible characters of $S_n$.  These irreducible characters $\chi^{\lambda}$ are well-known to be indexed by integer partitions $\lambda=(\lambda_1 \geq \lambda_2 \geq \cdots \geq \lambda_{\ell} \geq 0)$ with $|\lambda|:=\sum_i \lambda_i=n$.  We write $\Y$ for \emph{Young's lattice}, the set of partitions of all positive integers, and $\Y_n$ for the set of partitions of $n$.  Whenever $n-|\lambda|\geq \lambda_1$, we write $\lambda[n]$ for the partition $(n-|\lambda|, \lambda_1, \cdots, \lambda_{\ell})$ of $n$.

Our main result is Theorem~\ref{thm:main}, which shows that---although the functions $M_{\sigma, d, n}$ are not obviously related to one another as $n$ grows---in fact they ``stabilize" in a well-defined sense.  In particular, the values of $M_{\sigma,d,n}$ for all conjugacy classes in all symmetric groups $S_n$ may be determined knowing only the values in small symmetric groups.  The proof takes advantage of new techniques for understanding permutation patterns in conjugacy classes using partition algebras.

Partition algebras are well known to play a role in ``representation stability phenomena'', for example, they are endomorphism algebras in the Deligne category $\underline{\mathrm{Rep}}(S_t)$ (see \cite{ComesOstrik}). The Deligne category is a category depending on a parameter $t$ which ``interpolates'' the representation categories of symmetric groups $S_n$, and results about $\underline{\mathrm{Rep}}(S_t)$ imply results for all symmetric groups. Given that the partition algebras play a key role in the proof of our main result (Theorem \ref{thm:main}), it would be interesting to see whether there are other connections between permutation patterns and the active area of representation stability (see \cite{SamSnowden} for a survey of other aspects of representation stability).

As we will see, the family $M_{\sigma, d, n}$ behaves nicely as $n$ varies.  It will therefore sometimes be convenient to think of the family as defining a single function
\[
M_{\sigma,d}: \coprod_{n} S_n \to \N.
\]

\begin{theorem}
\label{thm:main}
Let $\sigma \in S_k$, and $d \geq 1$, then:
\begin{enumerate}[label=(\alph*)]
    \item \label{part:polynomial} $M_{\sigma, d}(\pi)$ is a polynomial in the variables $n, m_1(\pi), m_2(\pi), \ldots, m_{dk}(\pi)$ of degree at most $dk$, where $n$ has degree $1$ and $m_i$ has degree $i$. 
    \item \label{part:stable} For all $n \geq 2dk$ we have an equality of class functions
\[
M_{\sigma, d, n} = \sum_{\substack{\lambda \in \Y \\ |\lambda|\leq dk}} a^{\lambda}_{\sigma,d}(n) \chi^{\lambda[n]}
\]
for some family of polynomials $a^{\lambda}_{\sigma, d} \in \Q[n]$ of degree at most $dk-|\lambda|$.
\end{enumerate}
\end{theorem}

\begin{remark}
Theorem~\ref{thm:main} is far from obvious from the definitions; indeed it is not even clear a priori that $M_{\sigma, d}(\pi)$ should be determined by $n,m_1(\pi),\ldots,m_{dk}(\pi)$ alone.
\end{remark}

As a very special case of Theorem~\ref{thm:main}, we also recover a result of Zeilberger.

\begin{cor}[Zeilberger \cite{Zeilberger}]
\label{cor:recover-zeilberger}
The $d$-th moment of $N_{\sigma}$ on all of $S_n$ (rather than refined by conjugacy class, as in $M_{\sigma, d, n}$) is given by a polynomial in $n$, provided $n \geq 2dk$.
\end{cor}
\begin{proof}
By definition the desired moment is
\begin{align*}
    \frac{1}{n!}\sum_{\pi \in S_n} N_{\sigma}^d(\pi)&=\frac{1}{n!}\sum_{\substack{C \\ \text{conjugacy class}}} \sum_{\pi \in C} N_{\sigma}^d(\pi) \\
    &=\frac{1}{n!}\sum_{\pi \in S_n} M_{\sigma, d, n}(\pi) \\
    &=\langle M_{\sigma, d, n} , \mathbbm{1}_{S_n} \rangle, \\
    &= \sum_{\substack{\lambda \in \Y \\ |\lambda|\leq dk}} a_{\sigma, d}^\lambda(n) \langle \chi^{\lambda[n]} , \mathbbm{1}_{S_n} \rangle,
\end{align*}
where $\langle, \rangle$ denotes the inner product of class functions. For $n \geq 2dk$, this is just $a_{\sigma, d}^\varnothing(n)$. By Theorem~\ref{thm:main} this is a polynomial.
\end{proof}

Theorem~\ref{thm:main} also allows for the computation of moments of $N_{\sigma}$ on non-uniformly random permutations, such as products of random transpositions; as observed by Hultman \cite{Hultman2014}, this reduces to computing the character inner product of $M_{\sigma, d, n}$ and another character.  Since part \ref{part:stable} of the theorem implies that the support of $M_{\sigma, d, n}$ stabilizes, such inner products become tractable finite sums, whereas the analogous sums for permutation statistics other than $N_{\sigma}$ may gain more terms as $n$ increases.
\\
 
The remainder of the paper is organized as follows: Section~\ref{sec:partition algebras} contains background material on partition algebras.  Section~\ref{sec:proof-of-main} contains the proof of Theorem~\ref{thm:main}, which connects partition algebras to permutation patterns.

\subsection{Acknowledgements}

We are grateful to Axel Hultman for his help with important references and to Pavel Etingof for his helpful comments.

\section{Partition algebras}
\label{sec:partition algebras}

Partition algebras were first introduced by Martin \cite{Martin} and, independently, Jones \cite{Jones} in connection with statistical mechanics. The partition algebras satisfy a Schur--Weyl-type duality with the symmetric group, which has made the partition algebra an object relevant to the study of representations of symmetric groups (for example, see \cite{Deligne}, \cite{BHH} and \cite{BDO}). An excellent introduction to partition algebras can be found in \cite{ComesOstrik}; every result in this section is explained there, apart from Theorem~\ref{bicharacter_theorem}.

\begin{defin}
A \emph{set partition} $\{U_i\}$ of a set $X$ is a family of subsets $U_i$ of $X$ such that $X$ is the disjoint union of the $U_i$. We say that the $U_i$ are the \emph{parts} of the set partition $\{U_i\}$. We call a set partition of the set $\{1,\ldots, k, 1^\prime, \ldots, k^\prime\}$ a \emph{$(k,k)$-set partition}.
\end{defin}

It is standard to depict set partitions diagrammatically, as this allows for a concise definition of the partition algebras. Given a set partition $\{U_i\}$ of a set $X$, we may consider a graph whose vertices are given by elements of $X$, and whose connected components are precisely the parts $U_i$. Although there are many choices of edges yielding the same connected components, they will be equivalent for our purposes. We will not distinguish between a set partition and a graph representing it, so we will refer to the elements of $U_i$ as ``vertices''.

\begin{defin}
The \emph{partition algebra} $\Par_k(t)$ (where $k \in \mathbb{Z}_{\geq 0}$ and $t \in \mathbb{C}$), as a complex vector space, has a basis consisting of all $(k,k)$-set partitions. The product $P_1 P_2$ of two partitions is computed as follows: Relabel $P_2$ by adding a prime to each element of the underlying set (so it is a set partition of $\{1^\prime, \ldots, k^\prime, 1^{\prime\prime}, \ldots, k^{\prime\prime}\}$). Consider the graph $G(P_1,P_2)$ with vertex set
\[
\{1,\ldots,k,1^\prime, \ldots, k^\prime, 1^{\prime\prime}, \ldots, k^{\prime\prime}\},
\]
which is obtained by identifying each vertex in $\{1^\prime,  \ldots, k^\prime\}$ in $P_1$ with the vertex in $P_2$ bearing the same label (and retaining the edges in both $P_1$ and $P_2$).  Let $c(P_1, P_2)$ be the number of connected components of $G(P_1, P_2)$ consisting only of vertices from $\{1^\prime, \ldots, k^\prime \}$. Then
\[
P_1 P_2 = t^{c(P_1,P_2)} P_3,
\]
where $P_3$ is obtained by removing the elements $\{1^\prime, \ldots, k^\prime\}$ from the set partition corresponding to $G(P_1,P_2)$, and then relabelling all elements with double primes to have single primes.
\end{defin}

\begin{ex}
Suppose we take the $(3,3)$-set partitions $P_1 = \{\{1,1^\prime\},\{2^\prime\},\{2,3,3^\prime\}\}$ and $P_2 = \{\{1\},\{1^\prime\},\{2\},\{2^\prime,3\},\{3^\prime\}\}$ depicted below:
\[
P_1=
\begin{tikzcd}
1\arrow[d, dash, thick]& 2\arrow[r, dash, thick] & 3\arrow[d, dash, thick]\\
1^\prime & 2^\prime & 3^\prime 
\end{tikzcd}, \hspace{0.5in}
P_2=
\begin{tikzcd} 1& 2 & 3\arrow[ld,dash, thick]\\
1^\prime & 2^\prime & 3^\prime
\end{tikzcd}.
\]
Then the graph $G(P_1, P_2)$ is
\[
\begin{tikzcd}
1\arrow[d, dash, thick]& 2\arrow[r, dash, thick] & 3\arrow[d, dash, thick]\\
1^\prime & 2^\prime & 3^\prime\arrow[ld, dash, thick] \\
1^{\prime\prime} & 2^{\prime\prime} & 3^{\prime\prime}
\end{tikzcd},
\]
from which we read off the set partition $\{\{1,1^\prime\},\{1^{\prime\prime}\},\{2,3,3^\prime,2^{\prime\prime}\},\{2^\prime\},\{3^{\prime\prime}\}\}$.  Exactly one part consists only of primed vertices, namely $\{2^\prime\}$. Thus $c(P_1,P_2) = 1$. Taking the induced set partition on the top and bottom rows, we get
\[
P_3 = 
\begin{tikzcd}
1& 2\arrow[d, dash, thick]\arrow[r, dash, thick] & 3\\
1^\prime & 2^\prime & 3^\prime
\end{tikzcd},
\]
so we conclude $P_1 P_2 = t P_3$.
\end{ex}

\begin{remark}
The partition algebras arise as endomorphism algebras in the Deligne category $\underline{\mathrm{Rep}}(S_t)$, introduced in \cite{Deligne}. These categories ``interpolate'' the representation categories of all of the symmetric groups $S_n$ in a certain sense. We direct the interested reader to \cite{ComesOstrik}.
\end{remark}

The symmetric group $S_n$ tautologically acts on the set $[n] = \{1,2,\ldots, n\}$. We let $V$ be the associated permutation representation, although to avoid confusion we write $v_i$ for the basis vector corresponding to $i \in [n]$. We consider $V^{\otimes k}$, which is simply the permutation representation on the set $[n]^k$. It therefore has a basis consisting of words of length $k$ in the alphabet $[n]$. Given such a word $I = i_1 i_2 \cdots i_k$, we write $v_I = v_{i_1} \otimes v_{i_2} \otimes \cdots \otimes v_{i_k}$.

In analogy with the Schur--Weyl duality map $\mathbb{C}S_k \to \End_{GL(V)}(V^{\otimes k})$, there is a map
\[
\Phi_{k,n}: \Par_k(n) \to \End_{S_n}(V^{\otimes k}).
\]
Fix a set partition $P$ of the set $\{1, \ldots k, 1^\prime, \ldots, k^\prime\}$, and let $I \in [n]^k$. Then
\[
\Phi_{k,n}(P) \cdot v_I = 
\sum_{J \in [n]^k} C_{P}(J,I) v_{J},
\]
where $C_P(J,I)$ is either zero or one according the following rule: Give each unprimed vertex $r$ of $P$ the label $j_r$, and give each primed vertex $r^\prime$ of $P$ the label $i_r$; if each part of $P$ is monochromatic (i.e. if all edges have endpoints with the same label), then $C_P(J,I) = 1$, if not, $C_P(J,I)=0$. The action of $S_n$ on $V^{\otimes k}$ is by permuting the label set, and therefore respects monochromatic labellings. Thus, the action of $P$ commutes with the action of $S_n$.  One can check that $\Phi_{k,n}$ is an algebra map. The key point is that the calculation of ${P_1}{P_2}$ leads to considering monochromatic labellings of $G(P_1, P_2)$, and each of the $c(P_1, P_2)$ components that are removed in the calculation may be assigned an arbitrary label in $[n]$. This is how the scalar multiple $n^{c(P_1,P_2)}$ arises (note that we have set the parameter $t$ to equal $n$). 

Let $P_1$, $P_2$ be set partitions of a set $X$. We say that $P_1$ is a \emph{coarsening} of $P_2$ if whenever $x_1, x_2 \in X$ are in the same part of $P_2$, they are also in the same part of $P_1$. This makes set partitions of $X$ into a poset.

\begin{defin}
For a $(k,k)$-set partition $P$, let $x_P \in \Par_k(t)$ be defined by the recursion
\[
P=\sum_{\mbox{$P^\prime$ is a coarsening of $P$}} x_{P^\prime}.
\]
\end{defin}
It is immediate that the $x_P$ are a basis of $\Par_k(t)$ by an upper-unitriangularity argument. Moreover, the $x_P$ are linear combinations of the ${P^\prime}$ whose coefficients do not depend on $t$, as we have made no reference to the multiplicative structure.  An easy inclusion-exclusion argument determines $\Phi_{k,n}(x_P)$.

\begin{prop}
We have that
\[
\Phi_{k,n}(x_P) \cdot v_I = \sum_{J \in [n]^k} C_{P}^\prime(J,I) v_{J},
\]
where $C_{P}^\prime(J,I)$ is computed similarly to $C_{P}(J,I)$, except that it is only nonzero if in addition to being monochromatic, distinct parts of $P$ have distinct labels.
\end{prop}

The $x_P$ are useful because they control the kernel of $\Phi_{k,n}$:

\begin{theorem} \label{duality_hom_theorem}
The map $\Phi_{k,n}$ is always surjective. The kernel is spanned by $x_P$ where $P$ has strictly more than $n$ parts. In particular, $\Phi_{k,n}$ is injective whenever $n \geq 2k$.
\end{theorem}

We now perform a character calculation similar to that in Theorem 3.2.2 of \cite{Halverson}.

\begin{theorem} \label{bicharacter_theorem}
Let $P$ be a $(k,k)$-set partition (and hence an element of $\Par_k(n)$). Fix an element $g \in S_n$ of cycle type $(1^{m_1}2^{m_2}\cdots)$. Then, viewing $V^{\otimes k}$ as module for $\mathbb{C}S_n \otimes \Par_k(n)$,
\[
\tr_{V^{\otimes k}}(g \otimes P)
\]
is a polynomial in $n, m_1, m_2, \ldots, m_k$. If we let $n$ have degree 1, and $m_i$ have degree $i$, then this polynomial has degree at most $k$.
\end{theorem}

\begin{proof}
Let us consider the action of $g \otimes P$ on $v_I$ for $I \in [n]^k$:
\[
g \otimes P \cdot v_I = \sum_{J} C_P(J,I) v_{g(J)}.
\]
The coefficient of $v_I$ in this sum is equal to $C_P(g^{-1}(I),I)$, so
\[
\tr_{V^{\otimes k}}(g \otimes P) = \sum_{I} C_P(g^{-1}(I),I).
\]
We now count the $I \in [n]^k$ such that $C_P(g^{-1}(I),I) = 1$. Because of the monochromatic condition, it suffices to instead consider labellings of the parts of $P$, rather than the individual vertices. Consider the directed graph $G$ whose vertices are parts of $P$, having a directed edge from $U_i$ to $U_j$ if there is $r \in [n]$ such that $r \in U_i$ and $r^\prime \in U_j$ (note that this graph may have loops). In order to have $C_P(g^{-1}(I),I) = 1$, it is necessary and sufficient that whenever there is an edge from $U_i$ to $U_j$, the label of $U_j$ must be obtained by applying $g$ to the label of $U_i$. Note that because the action of $g$ on labels is invertible, the label of $U_j$ also determines the label of $U_i$.

We now count the labellings.  For each connected component in $G$, choosing the label of any vertex determines the labels for all vertices in that component. This may be done algorithmically by depth-first search; however, if we encounter an edge connecting two vertices that have already been visited, we still require the labellings to be related by $g$. This happens precisely when there is a cycle in the graph. If such a cycle has $l_+$ edges directed forwards and $l_-$ edges directed backwards, we require that $g^{l_+-l_-}$ must fix the label of every vertex in the cycle, and hence the label of every vertex in the component. Let $l_{tot}$ be the greatest common divisor of $l_+ - l_-$ among all undirected cycles in the connected component in $G$, where $l_{tot}=0$ if there are no cycles. The the label of a vertex must be fixed by $g^{l_{tot}}$ and the number of labels satisfying this condition is
\[
Q_{l_{tot}} = \sum_{d \mid l_{tot}} d m_d,
\]
which is equal to $n$ if $l_{tot} = 0$. So the total number of $I \in [n]^k$ is equal to the product of $Q_{l_{tot}}$ across all connected components of the graph $G$.

It remains to show that $m_i$ with $i > k$ cannot appear, and to analyse the degree of the polynomial obtained in this way. Each element $r \in [n]$ contributes an edge to $G$ (unless an edge was already present). This means that $G$ has at most $n$ edges, so any component with a cycle must have a cycle of length at most $n$, meaning that $l_{tot} \leq n$. This shows that only the claimed variables appear. It also shows that a component of $G$ with $r$ edges contributes a multiplicative factor of degree at most $r$. Note that $r=0$ cannot happen: every part contains either a primed or unprimed vertex, and therefore has an incoming or outgoing edge in $G$ (even if it connects back to the same part).
\end{proof}

\section{Proof of the stability theorem}
\label{sec:proof-of-main}

We begin by constructing an algebraic method for counting permutation patterns. In this section, $V$ is the permutation representation of $S_n$ corresponding to the natural action on $[n]$, and $v_I$ is the pure tensor in $V^{\otimes k}$ corresponding to the word $I \in [n]^k$.  We extend the notion of pattern containment from permutations to more general functions and tuples as follows: Let $\sigma$ be any function from the set $\{1,\ldots,k\}$ to itself (not necessarily a permutation), then we say that the tuple $(i_1, \ldots, i_k)$ is \emph{$\sigma$-sorted} if $i_a < i_b$ if and only if $\sigma(a) < \sigma(b)$.

\begin{defin}
For a function $\sigma :[k] \to [k]$, let $E_\sigma: \mathbb{C} \to V^{\otimes k}$ be the linear map defined by
\[
E_\sigma(1) = \sum_{\substack{I \in [n]^k \\I \mbox{ is $\sigma$-sorted}}} v_I.
\]
Similarly, we define the transpose map $E_\sigma^T: V^{\otimes k} \to \mathbb{C}$ by
\[
E_\sigma^T(v_I) = 
\begin{cases} 
1 & \mbox{$I$ is $\sigma$-sorted} \\
0 & \mbox{otherwise} 
\end{cases}.
\]
\end{defin}

\begin{lemma}
For $g \in S_n$,  we have $N_\sigma(g) = \tr_\mathbb{C}(E_{\sigma}^T g E_{\Id})$. 
\end{lemma}

\begin{proof}
We compute the value of $E_{\sigma}^T g E_{\Id}$ at $1 \in \mathbb{C}$. Note that $E_{\Id}(1)$ is the sum of pure tensors $v_I$ such that the word $I$ is strictly increasing. Let us consider a single such term $v_{i_1} \otimes v_{i_2} \otimes \cdots \otimes v_{i_k}$. Applying $g$ gives $v_{g(i_1)} \otimes v_{g(i_2)} \otimes \cdots \otimes v_{g(i_k)}$. Finally, applying $E_\sigma^T$ yields one if $(g(i_1), g(i_2), \ldots, g(i_k))$ is $\sigma$-sorted, which is to say that $I$ is an occurrence of $\sigma$ in $g$, and zero otherwise.
\end{proof}

\begin{lemma} \label{shuffle_lemma}
Given $\alpha: [k_1] \to [k_1]$ and $\beta: [k_2] \to [k_2]$, we have
\[
E_\alpha \otimes E_\beta = \sum_{\gamma} E_\gamma,
\]
where the sum is over $\gamma: [k_1 + k_2] \to [k_1 + k_2]$ such that $\gamma$ restricted to $\{1, \ldots, k_1\}$ is $\alpha$-sorted, $\gamma$ restricted to $\{k_1+1, \ldots, k_1+k_2\}$ is $\beta$-sorted, and the image of $\gamma$ is of the form $\{1,\ldots, r\}$ for some $r$.
\end{lemma}
\begin{proof}
By definition, $E_\alpha \otimes E_\beta$ sends $1 \in \mathbb{C}$ to the sum of all $v_I \otimes v_J$ such that $I$ is $\alpha$-sorted and $J$ is $\beta$-sorted. So, we obtain precisely the $v_{K}$ with the first $k_1$ elements $\alpha$-sorted, and final $k_2$ elements $\beta$-sorted (each with multiplicity 1). The stipulation on the image of $\gamma$ is to avoid overcounting (for example, a $(1,1)$-pattern is the same thing as a $(2,2)$-pattern, so $E_{(1,1)} = E_{(2,2)}$). Note that $E_{\gamma} = E_{\gamma^\prime}$ if and only if $\gamma$ is $\gamma^\prime$-sorted and vice-versa. If $r$ is the size of the image of $\gamma$, there is a unique $\gamma$-sorted element of $[r]$. This completes the proof.
\end{proof}

\begin{theorem}\label{averaging_theorem}
Given $\alpha, \beta$ functions $[k] \to [k]$, the linear map $T_{\alpha, \beta} : V^{\otimes k} \to V^{\otimes k}$ given by
\[
\frac{1}{n!}\sum_{x \in S_n} x^{-1} E_{\alpha} E_\beta^T x
\]
may be expressed in the form
\[
\Phi_{k,n}\left(\sum_{P} a_{P} {P}\right),
\]
where $P$ are $(k,k)$-set partitions, and the $a_P \in \mathbb{C}$ are independent of $n$.
\end{theorem}

\begin{proof}
Because $T_{\alpha, \beta}$ is defined by averaging over $S_n$, it commutes with the action of $S_n$ on $V^{\otimes k}$. It therefore lies in the image of $\Phi_{k,n}$, and so for each $n$, we have
\[
T_{\alpha, \beta} = \Phi_{k,n}\left(\sum_{P} a_P(n) P \right).
\]
We must show that $a_P(n)$ may be taken to be independent of $n$. It is equivalent to show that 
\[
T_{\alpha, \beta} = \Phi_{k,n}\left(\sum_{P} b_P(n) x_P \right),
\]
with $b_P(n)$ independent of $n$. This is easier to check, because of the property that for any $I, J \in [n]^k$ there is a unique $(k,k)$-set partition $P$ such that $v_J$ appears with nonzero coefficient in $\Phi_{k,n}(x_P) v_I$. This $P$ is obtained by labelling the unprimed vertices in $\{1,\ldots, k, 1^\prime, \ldots, k^\prime\}$ with $J$, and labelling the primed vertices with $I$, then taking the parts of $P$ to be the maximal monochromatic subsets of the vertices. One can easily reverse this construction to find such $I$ and $J$ given $P$, provided the number of labels available (i.e. $n$) is at least the number of parts of $P$.

Let us choose $P$ with at most $n$ parts, and then select $I$ and $J$ as such that $v_J$ appears in $\Phi_{k,n}(x_P) v_I$. Let $L$ be the set of labels appearing in at least one of $I$ and $J$. Note that $|L|$ is precisely the number of parts of $P$.

The coefficient of $v_J$ in
\[
\frac{1}{n!}\sum_{x \in S_n} x^{-1} E_{\alpha} E_\beta^T x v_I
\]
is
\[
\frac{1}{n!} |\{x \in S_n \mid x(I)\mbox{ is $\beta$-sorted and } x(J) \mbox{ is $\alpha$-sorted} \}|.
\]
The sorting condition on $x$ is determined by how it reorders the elements of $L$. If $S_L$ is the subgroup of $S_n$ permuting the subset $L$ of $[n]$, each coset of $S_L$ yields each possible relative ordering of the elements of $L$ exactly once. Hence each coset has the same number of elements satisfying the sorting condition. Thus the coefficient becomes
\begin{eqnarray*}
& &
\frac{1}{n!} |S_n/S_L||\{x \in S_L \mid x(I)\mbox{ is $\beta$-sorted and } x(J) \mbox{ is $\alpha$-sorted} \}|\\
&=&
\frac{1}{|L|!}
|\{x \in S_L \mid x(I)\mbox{ is $\beta$-sorted and } x(J) \mbox{ is $\alpha$-sorted} \}|.
\end{eqnarray*}
This coefficient is independent of $n$. We note that if $|L| > n$ (so that it would be impossible to choose suitable $I$ and $J$), $\Phi_{k,n}(x_P) = 0$, so we may take $b_P$ to have the computed value even for small $n$.
\end{proof}

We are now ready to prove the main theorem.

\begin{proof}[Proof of Theorem~\ref{thm:main}]
We first prove part \ref{part:polynomial}.  We compute:
\begin{eqnarray*}
M_{\sigma, d, n}(\pi) &=&
\frac{1}{|C_{\pi}|} \sum_{\pi' \in C_{\pi}} N_{\sigma}^d(\pi') \\
&=&
\frac{1}{n!} \sum_{x \in S_n} \tr_{\mathbb{C}}(E_\sigma^T x\pi x^{-1} E_{\Id} )^d \\
&=&
\frac{1}{n!} \sum_{x \in S_n} \tr_{V^{\otimes k}}(x^{-1} E_{\Id} E_\sigma^T x\pi  )^d \\
&=&
\frac{1}{n!} \sum_{x \in S_n} \tr_{(V^{\otimes k})^{\otimes d}}(x^{-1} (E_{\Id})^{\otimes d} (E_\sigma^T)^{\otimes d} x \pi ).
\end{eqnarray*}
Now we note that $E_{\Id}^{\otimes d}$ is a linear combination of $E_{\alpha}$ for some patterns $\alpha$ by Lemma~\ref{shuffle_lemma}, and similarly $(E_{\sigma}^{T})^{\otimes d}$ is a linear combination of $E_{\beta}^T$ for some patterns $\beta$. So our expression becomes a linear combination of
\begin{eqnarray*}
\frac{1}{n!} \sum_{x \in S_n} \tr_{(V^{\otimes dk})}(x^{-1} (E_{\alpha}) E_\beta^T xg ) &=&
\tr_{(V^{\otimes dk})}\left(\frac{1}{n!} \sum_{x \in S_n} x^{-1} (E_{\alpha}) E_\beta^T x \pi \right) \\
&=&
\tr_{(V^{\otimes dk})}\left(\Phi_{dk, n}\left(\sum_{P}{a_P} D \right)\otimes \pi \right). 
\end{eqnarray*}
Here we have used Theorem~\ref{averaging_theorem}. Finally, Theorem~\ref{bicharacter_theorem} guarantees that the result is a polynomial of the claimed form; this establishes part \ref{part:polynomial}.

For $\mu \in \Y_r$ it is well known (see Chapter 1, Section 7, Example 14 of Macdonald's book \cite{Macdonald}) that the function $\chi^{\mu[n]}$ on $\coprod_{n \geq 2r} S_n$ agrees with a polynomial $P_{\mu}(m_1,\ldots,m_r)$ of degree $r$, where $m_i$ has degree $i$ for each $i=1,\ldots,r$.  As in Section~\ref{sec:intro}, we evaluate these \emph{character polynomials} on a permutation $\pi$ by specifying the $m_i$ so that $1^{m_1}2^{m_2} \cdots$ is the cycle type of $\pi$.  Since the space of polynomials of degree at most $r$ (under this grading) has dimension $\sum_{n=0}^r |\Y_n|$, and since the character polynomials must be linearly independent by the orthogonality relation for characters, the character polynomials $\{P_{\mu}$ for $\mu \in \Y_n, n \leq r\}$ must form a basis of this space.

Now, by part \ref{part:polynomial}, we have that $M_{\sigma, d}$ is a polynomial in the variables $n,m_1,\ldots,m_{dk}$ of degree at most $dk$, where $\sigma \in S_k$ and $n$ has degree 1.  Decomposing $M_{\sigma, d}$ by applying the spanning property of character polynomials from the previous paragraph implies \ref{part:stable}.
\end{proof}

\bibliographystyle{plain}
\bibliography{paper.bib}

\begin{thebibliography}{10}

\bibitem{Abe-Billey}
Hiraku Abe and Sara Billey.
\newblock Consequences of the {L}akshmibai-{S}andhya theorem: the ubiquity of
  permutation patterns in {S}chubert calculus and related geometry.
\newblock In {\em Schubert calculus---{O}saka 2012}, volume~71 of {\em Adv.
  Stud. Pure Math.}, pages 1--52. Math. Soc. Japan, [Tokyo], 2016.

\bibitem{Archer-Elizalde}
Kassie Archer and Sergi Elizalde.
\newblock Cyclic permutations realized by signed shifts.
\newblock {\em J. Comb.}, 5(1):1--30, 2014.

\bibitem{BHH}
Georgia Benkart, Tom Halverson, and Nate Harman.
\newblock Dimensions of irreducible modules for partition algebras and tensor
  power multiplicities for symmetric and alternating groups.
\newblock {\em Journal of Algebraic Combinatorics}, 46(1):77--108, 2017.

\bibitem{Bona-textbook}
Mikl\'{o}s B\'{o}na.
\newblock {\em Combinatorics of permutations}.
\newblock Discrete Mathematics and its Applications (Boca Raton). CRC Press,
  Boca Raton, FL, second edition, 2012.
\newblock With a foreword by Richard Stanley.

\bibitem{Bona-Cory}
Mikl\'{o}s B\'{o}na and Michael Cory.
\newblock Cyclic permutations avoiding pairs of patterns of length three.
\newblock {\em Discrete Math. Theor. Comput. Sci.}, 21(2):Paper No. 8, 15,
  2019.

\bibitem{invol2}
Mikl\'{o}s B\'{o}na, Cheyne Homberger, Jay Pantone, and Vincent Vatter.
\newblock Pattern-avoiding involutions: exact and asymptotic enumeration.
\newblock {\em Australas. J. Combin.}, 64:88--119, 2016.

\bibitem{Bona-Smith}
Mikl\'{o}s B\'{o}na and Rebecca Smith.
\newblock Pattern avoidance in permutations and their squares.
\newblock {\em Discrete Math.}, 342(11):3194--3200, 2019.

\bibitem{BDO}
Christopher Bowman, Maud De~Visscher, and Rosa Orellana.
\newblock The partition algebra and the kronecker coefficients.
\newblock {\em Transactions of the American Mathematical Society},
  367(5):3647--3667, 2015.

\bibitem{ComesOstrik}
Jonathan Comes and Victor Ostrik.
\newblock {On blocks of Deligne's category $\underline{\mathrm{Rep}}(S_t)$}.
\newblock {\em Advances in Mathematics}, 226(2):1331--1377, 2011.

\bibitem{Deligne}
Pierre Deligne.
\newblock {La cat{\'e}gorie des repr{\'e}sentations du groupe sym{\'e}trique
  $S_t$, lorsque t n’est pas un entier naturel}.
\newblock {\em Algebraic groups and homogeneous spaces}, 19:209--273, 2007.

\bibitem{invol1}
W.~M.~B. Dukes, V\'{\i}t Jel\'{\i}nek, Toufik Mansour, and Astrid Reifegerste.
\newblock New equivalences for pattern avoiding involutions.
\newblock {\em Proc. Amer. Math. Soc.}, 137(2):457--465, 2009.

\bibitem{Gill2013}
Jonna Gill.
\newblock The k-assignment polytope, phylogenetic trees, and permutation
  patterns.
\newblock {\em Ph.D. Thesis at Link\"oping University}, pages 103--125, 2013.

\bibitem{Halverson}
Tom Halverson.
\newblock Characters of the partition algebras.
\newblock {\em Journal of Algebra}, 238(2):502--533, 2001.

\bibitem{Hultman2014}
Axel Hultman.
\newblock Permutation statistics of products of random permutations.
\newblock {\em Adv. in Appl. Math.}, 54:1--10, 2014.

\bibitem{Janson-Nakamura-Zeilberger}
Svante Janson, Brian Nakamura, and Doron Zeilberger.
\newblock On the asymptotic statistics of the number of occurrences of multiple
  permutation patterns.
\newblock {\em J. Comb.}, 6(1-2):117--143, 2015.

\bibitem{Jones}
Vaughan~FR Jones.
\newblock {The Potts model and the symmetric group}.
\newblock {\em Subfactors (Kyuzeso, 1993)}, pages 259--267, 1994.

\bibitem{Knuth}
Donald~E. Knuth.
\newblock {\em The art of computer programming. {V}ol. 1: {F}undamental
  algorithms}.
\newblock Second printing. Addison-Wesley Publishing Co., Reading,
  Mass.-London-Don Mills, Ont, 1969.

\bibitem{Macdonald}
I.~G. Macdonald.
\newblock {\em Symmetric functions and {H}all polynomials}.
\newblock Oxford Mathematical Monographs. The Clarendon Press, Oxford
  University Press, New York, second edition, 1995.
\newblock With contributions by A. Zelevinsky, Oxford Science Publications.

\bibitem{Martin}
Paul~P Martin.
\newblock {Representations of graph Temperley-Lieb algebras}.
\newblock {\em Publications of the Research Institute for Mathematical
  Sciences}, 26(3):485--503, 1990.

\bibitem{SamSnowden}
STEVEN~V SAM and ANDREW SNOWDEN.
\newblock {STABILITY} {PATTERNS} {IN} {REPRESENTATION} {THEORY}.
\newblock {\em Forum of Mathematics, Sigma}, 3, jun 2015.

\bibitem{Zeilberger}
Doron Zeilberger.
\newblock Symbolic moment calculus. {I}. {F}oundations and permutation pattern
  statistics.
\newblock {\em Ann. Comb.}, 8(3):369--378, 2004.

\end{thebibliography}
\end{document}